\documentclass[12pt]{amsart}

\usepackage[margin=1in]{geometry}
\usepackage{graphicx}
\usepackage{amsthm, amsmath, amssymb, tikz, bbm,amsfonts,comment}
\usepackage{mathtools}
\usepackage{xifthen}
\usepackage{hyperref}

\usepackage[shortlabels]{enumitem}
\usepackage{todonotes}
\usetikzlibrary{calc,shapes, backgrounds}

\newcommand{\op}{\operatorname}

\newcommand{\diam}{sdiam}

\newcommand{\old}[1]{{}}

\newtheorem{defn}{Definition}

\newtheorem{thm}[defn]{Theorem}
\newtheorem{lem}[defn]{Lemma}

\newtheorem{cor}[defn]{Corollary}

\begin{document}

\title{The Steiner distance problem for large vertex subsets in the hypercube}

\author{\'Eva Czabarka}
\author{Josiah Reiswig}
\author{L\'aszl\'o Sz\'ekely}

\address{\'Eva Czabarka\\Department of Mathematics \\ University of South 
Carolina \\ Columbia SC 29212 \\ USA
\and Visiting Professor\\ Department of Mathematics and Applied Mathematics\\ 
University of Johannesburg\\
P.O. Box 524, Auckland Park, Johannesburg 2006\\South Africa}
\email{czabarka@math.sc.edu}


\address{Josiah Reiswig\\ Department of Mathematics \\ Anderson University
\\ Anderson SC 29621 \\ USA}
\email{jreiswig@andersonuniversity.edu}

\address{L\'aszl\'o Sz\'ekely\\Department of Mathematics \\ University of South 
Carolina \\ Columbia SC 29212 \\ USA
\and Visiting Professor\\ Department of Mathematics and Applied Mathematics\\ 
University of Johannesburg\\
P.O. Box 524, Auckland Park, Johannesburg 2006\\ South Africa}
\email{szekely@math.sc.edu}




\thanks{
The last two authors were  supported in part by the National Science Foundation 
contract DMS-1600811.}

\begin{abstract}
We find the asymptotic behavior of the Steiner k-diameter
of the $n$-cube if $k$ is large. Our main contribution is the lower bound, which
utilizes the probabilistic method.
\end{abstract}

\maketitle

\section{Introduction}
For a connected graph $G$ of order at least 2 and $S\subseteq V(G)$, the {\sl 
Steiner distance} $d(S)$ among the vertices of $S$ is the minimum size among all 
connected subgraphs whose vertex sets contain $S$. Necessarily, such a minimum
subgraph must be a tree and such a tree is called a \textit{Steiner tree}. The Steiner
distance was introduced by G. Chartrand, O.R. Oellermann, S. Tian and H.B. Zou 
\cite{definition}, and 
it has turned into a well-studied parameter of graphs. Tao Jiang, Zevi Miller, 
and Dan Pritikin \cite{fixedk} studied how large  the Steiner distance of $k$ 
vertices  can be in the $n$-dimensional hypercube $Q_n$ as $n\rightarrow 
\infty$, while Zevi Miller and Dan Pritikin \cite{weightklayer} gave near tight 
bounds for the Steiner distance of a layer, i.e. vertices with the same number 
of 1's, in the $n$-dimensional hypercube $Q_n$ as $n\rightarrow \infty$. For a 
given $2\leq k\leq n$, the \textit{Steiner $k$-diameter} of the $n$-cube, 
$sdiam_k(Q_n)$, is the maximum Steiner distance among all $k$ subsets of 
$V(Q_n)$.

In this note we give natural upper bounds for the Steiner distance of a {\sl 
large} vertex set in the hypercube. It turns out that even the second order term 
in this estimate is close to tight. With these bounds, we determine 
$sdiam_k(Q_n)$ asymptotically for large $k$.

\section{Upper Bound}

For the upper bound, we utilize connected dominating sets of $Q_n$. A set 
$S\subset V(Q_n)$ is a \textit{dominating set} of $Q_n$ if every vertex of $Q_n$ 
is either an element of $S$ or has a neighbor in $S$. The minimum size of all 
dominating sets is called the \textit{domination number} of $Q_n$ and is denoted 
$\gamma(Q_n)$. The \textit{connected domination number}, denoted by 
$\gamma_c(Q_n)$, is minimum size of all connected dominating sets.

In 1988, Kabatyanskii and Panchenko \cite{kp} showed 
$\lim_{n\rightarrow\infty}\frac{\gamma(Q_n)}{2^n/n}=1$.
In an upcoming paper, Griggs \cite{griggs} 
utilizes this result to show that 
$\lim_{n\rightarrow\infty}\frac{\gamma_c(Q_n)}{2^n/n}=1$.
We use this last result to develop an upper bound for the Steiner diameter
of subsets of $V(Q_n)$. 

\begin{lem}\label{ub}
Suppose that $S\subset V(Q_n)$. Then, 
$$
d(S)\leq |S|+\dfrac{2^n}{n}(1+o(1)).
$$ 
\end{lem}

\begin{proof}
Begin with a minimum connected dominating set of $Q_n$.
Simply connect each of the
elements of $S$ to this connected dominating set. The resulting subgraph spans 
$S$ and contains at most $|S| + \gamma_c(Q_n) - 1$ edges. Using \cite{griggs}, 
we then have that
$
d(S)\leq |S|+\frac{2^n}{n}(1+o(1)). 
$
\end{proof}

\section{Lower Bound}
To bound the Steiner distance of large vertex subsets of $Q_n$ from below,
we partition the
vertices of the hypercube into two sets. Identifying each vertex of $Q_n$ into
a binary string of length $n$, we let vertices with an even number of 1's make up
the set of even vertices and denote this set by $\mathcal{E}_n$. 
Similarly, we let the vertices with an odd number of 1's make up
the set of odd vertices and denote this set by $\mathcal{O}_n$.
We refer to changing the value of the $i$'th entry of a binary 
string $v=v_0\cdots v_i\cdots v_n$ as ``flipping'' the $i$'th entry
of $v$. Given an entry $v_i$, we let $\bar{v}_i=1-v_i$. That is, $\bar{v_i}$
 is the flipped value of $v_i$. For the proof of Theorem~\ref{lb}, we use 
 probabilistic methods similar to those found in \cite{as}.

\begin{thm}\label{lb}
Suppose that $S \subset \mathcal{E}_n$, i.e., each vertex in 
$S$ contains an even number of 1's. Then,
$$
d(S)\geq |S|+\dfrac{|S|^2}{n2^n}-\frac{(n+1)}{2}.
$$
\end{thm}

\begin{proof}
Suppose that $S$ is a subset of the set of even vertices of $Q_n$. Let $\bar{S}$
be some subset of the odd vertices which is the image of $S$ under
some automorphism of $Q_n$. That is, $S\subset \mathcal{E}_n$,
$\bar{S}\subset \mathcal{O}_n$, and $\bar{S}=\gamma(S)$ for some
$\gamma\in \op{Aut}(Q_n)$.
Such a subset exists. Indeed, consider the set of
all vertices in $S$ with the first entry flipped.  Since $S$ and $\bar{S}$ are
isomorphic, we have that $d(S)=d(\bar{S})$. Suppose that $T=(V(T),E(T))$ and 
$\bar{T}=(V(\bar{T}),E(\bar{T}))$ are Steiner trees of $S$ and $\bar{S}$, 
respectively. Naively, we have that $d(S\cup \bar{S})\geq 2|S|-1$. Furthermore, 
connecting $T$ and $\bar{T}$ with at most $n$ edges
yields a subgraph of $Q_n$ connecting $S\cup \bar{S}$. Hence, 
$$
d(S\cup \bar{S})\leq |E(T)\cup E(\bar{T})|+n.
$$
Putting these inqualities together and applying the principle of inclusion and 
exclusion, we have 
\begin{align*}
2|S|-1&\leq |E(T) \cup E(\bar{T})|+n\\
&=|E(T)|+|E(\bar{T})|-|E(T)\cap E(\bar{T})|+n\\
&=2d(S)-|E(T)\cap E(\bar{T})|+n,
\end{align*} 
which implies that 
$$
2d(S)-|E(T)\cap E(\bar{T})|\geq 2|S|-(n+1).
$$
Note that if $\lambda_1$ and $\lambda_2$ are automorphisms of
$Q_n$ which preserve the parity of their inputs, then the inequality above 
extends to
\begin{equation}
2d(S)-|E(\lambda_1(T))\cap E(\lambda_2(\bar{T}))|\geq 2|S|-(n+1),\label{probeq}
\end{equation}
where $\lambda_i(T)$ and $\lambda_i(\bar{T})$ are the images of $T$ and $\bar{T}$
under $\lambda_i$ for $1\leq i\leq 2$.

Let $\Gamma=\langle\alpha, \beta_{i,j}:1\leq 0<j\leq n-1 \rangle$ be the subgroup 
of the group of automorphisms of $Q_n$ generated by the automorphisms
\begin{align*}
\alpha:& v_0v_1\cdots v_{n-1}\mapsto v_1\cdots v_{n-1}v_0\\
\beta_{i,j}:& v_0v_1\cdots v_i \cdots v_j\cdots v_{n-1}\mapsto  v_0v_1\cdots 
\bar{v_i} \cdots \bar{v_j}\cdots v_{n-1}.
\end{align*}
In words, $\alpha$ shifts each entry of its input to the left by 1 (modulo $n$),
while $\beta_{i,j}$ flips only the values of the $i$'th and $j$'th entries of its input.
Note that each element of $\Gamma$ preserves the parity of its input. 
We now verify the following claim:\\

\noindent \textbf{Claim:} For any two edges $e_1,e_2\in E(Q_n)$, there exists a 
\textit{unique} element of $\lambda\in \Gamma$ such that $\lambda(e_1)=e_2$.  \\

Suppose that $e_1=ab$ and 
$e_2=uv$ where $a$ and $u$ are even vertices while $b$ and $v$ are odd vertices. 
Without loss of generality, we may assume that $a=\mathbf{0}$, the vertex of all 
zeros. This implies that the string $b$ contains a single 1. We shall first 
prove existence of an automorphism $\lambda\in \Gamma$ mapping $e_1$ to $e_2$.

Since $u\in \mathcal{E}_n$, using a composition of automorphisms of the form
$b_{i,j}$ we may map 
$uv$ to $\mathbf{0} \hat{v}$, where $\hat{v}$ has a single 1.
Then, using some power of the automorphism $\alpha$, we 
may then map the edge $\mathbf{0}\hat{v}$ to the edge $\mathbf{0} b=e_1$. Let 
$\lambda$ be these composition of autmorphisms in $\Gamma$.  

To show that this automorphism is unique, we show that $|\Gamma|=n2^{n-1}$.
Since $\alpha\circ\beta_{ij}=\beta_{i-1,j-1}\circ\alpha$ (where the indexes are taken modulo $n$),
any $\lambda\in\Gamma$ can be described as first applying an appropriate power of
$\alpha$ and then flipping an even number of digits. As we have $n$ choices for the
power of $\alpha$ and $2^{n-1}$ choices for the subset of digits we flip, 
$|\Gamma|=n2^{n-1}$. Since $Q_n$ has $n2^{n-1}$ edges, and any 
$\lambda\in\Gamma$ maps the edge $\mathbf{0}b$ to one of these
in such a way that $0$ gets mapped to the endvertex in $\mathcal{E}_n$,
and all edges of $Q_n$ will be the image of $\mathbf{0}b$ under some
$\lambda\in\Gamma$, the claim follows.

We now consider the experiment of selecting elements $\lambda_1,\lambda_2\in 
\Gamma$ independently with uniform probability, and applying them to $T$ and 
$\bar{T}$, respectively. Consider the random variable $X=|E(\lambda_1(T))\cap 
E(\lambda_2(\bar{T}))|$. For the expected value of $X$, $\mathbb{E}(X)$,
we have that 
$$
\max\limits_{\lambda_1,\lambda_2}\{|E(\lambda_1(T)\cap 
\lambda_2(\bar{T}))|\}\geq \mathbb{E}(X).
$$
Using our claim, we observe that 
\begin{align*}
\mathbb{E}(X)
&=\sum_{f\in E(Q_n)}P[(f\in \lambda_1(T)) \text{ and } (f\in \lambda_2(\bar{T}))]\\
&=\sum_{f\in E(Q_n)}\frac{|E(T)|}{n2^{n-1}} \cdot 
\frac{|E(\bar{T})|}{n2^{n-1}}\\
&=\frac{|E(T)|^2}{n2^{n-1}}\\
&=\frac{d(S)^2}{n2^{n-1}},
\end{align*}
which implies 
$$
\max\limits_{\lambda_1,\lambda_2}\{|E(\lambda_1(T))\cap
E(\lambda_2(\bar{T}))\}\geq \frac{d(S)^2}{n2^{n-1}}.
$$
Using $\lambda_1$ and $\lambda_2$ which achieve this maximum and
applying inequality \eqref{probeq}, we see that
$$
2d(S)-\frac{d(S)^2}{n2^{n-1}}\geq 2|S|-(n+1).
$$
We are going to bootstrap the calculation above. Assume without loss of generality that $d(S) = |S| + x$ for some $x>0$. So, 
\begin{align*}
2(|S|+x)-\frac{(|S|+x)^2}{n2^{n-1}}
&\geq 2|S|-(n+1)\\
2|S|+2x-\frac{|S|^2+2|S|x+x^2}{n2^{n-1}}
&\geq 2|S|-(n+1)\\
2x-\frac{2|S|x}{n2^{n-1}}+2|S|-\frac{|S|^2+x^2}{n2^{n-1}}
&\geq 2|S|-(n+1)\\
2x(1-\frac{|S|}{n2^{n-1}})
&\geq \frac{|S|^2+x^2}{n2^{n-1}}-(n+1)\\
x
&\geq \frac{|S|^2}{n2^{n}}-\frac{(n+1)}{2},
\end{align*}
and the result is proven.
\end{proof}

With these results in hand, we can determine the asymptotic growth 
of $sdiam_k(Q_n)$ as for large $k$. In particular, we can determine the first and 
second order terms if $k=\Omega(2^n)$, while we can determine the first order
term if $2^n/n=o(k)$.

\begin{cor}
If $k=k(n)$, then
\begin{enumerate}
\item If $k=\Omega(2^n)$, then $\diam_k(Q_n)=k+\Theta(2^n/n)$, and

\item If ${2^n/n}=o(k)$, then 
$\lim\limits_{n\rightarrow\infty}\frac{sdiam_k(Q_n)}{k}=1$.
\end{enumerate}
\end{cor}

\begin{proof}
If $k\leq 2^{n-1}$, let $S\subset V(Q_n)$ be a subset of the even vertices of 
size $k$. If $k>2^{n-1}$, let $S$ contain all even vertices and choose the 
remaining odd vertices randomly. Applying the bounds determined in Lemma~\ref{ub}
and Theorem~\ref{lb}, we see that 
$$
k+\frac{k^2}{n2^n}-\frac{n+1}{2}\leq d(S)\leq \diam_k(Q_n) \leq k+\frac{2^n}{n}(1+o(1)).
$$
If $k=\Omega(2^n)$, $\diam_k(Q_n)$ is bounded above and below
by $k+\Theta(2^n/n).$
Alternatively, if $2^n/n=o(k)$, we have $\diam(Q_n)=k(1+o(1))$, giving 
$\lim_{n\rightarrow\infty}\frac{\diam_k(Q_n)}{k}=1$.
\end{proof}

\end{document}